\documentclass{amsart}
\usepackage{graphicx}
\newtheorem{Thm}{Theorem}
\newtheorem{Lem}[Thm]{Lemma}
\newtheorem{Prop}[Thm]{Proposition}
\newtheorem{Cor}[Thm]{Corollary}
\newtheorem{Ques}[Thm]{Question}
\begin{document}
\title{Lips and swallow-tails of singularities of product maps}
\author{Kazuto Takao}
\address{Institute of Mathematics for Industry, Kyushu University, 744 Motooka, Nishi-ku, Fukuoka 819-0395, Japan}
\email{takao@imi.kyushu-u.ac.jp}
\thanks{Supported by JSPS and CAPES under the Japan--Brazil Research Cooperative Program}
\begin{abstract}
Lips and swallow-tails are generic local moves of singularities of a smooth map to a $2$-manifold.
We prove that these moves of singularities of the product map of two functions on a $3$-manifold can be realized by isotopies of the functions.
\end{abstract}
\maketitle

\section{Introduction}

We consider the relationship between a pair of maps and the product map.
Let $M,P,Q$ be smooth manifolds and let $C^\infty (M,*)$ denote the space of smooth maps from $M$ to a smooth manifold $*$ endowed with the Whitney $C^\infty $ topology.
Two smooth maps $F\in C^\infty (M,P)$ and $G\in C^\infty (M,Q)$ determine the product map $(F,G)\in C^\infty (M,P\times Q)$ by $(F,G)(p)=(F(p),G(p))$.
Conversely, a smooth map $\varphi \in C^\infty (M,P\times Q)$ can be decomposed into $\pi _P\circ \varphi \in C^\infty (M,P)$ and $\pi _Q\circ \varphi \in C^\infty (M,Q)$, where $\pi _P:P\times Q\rightarrow P$ and $\pi _Q:P\times Q\rightarrow Q$ are the projections.
By \cite[Chapter~I\hspace{-.1em}I, Proposition~3.6]{golubitsky-guillemin}, this correspondence gives the homeomorphism
\begin{equation}
C^\infty (M,P)\times C^\infty (M,Q)\cong C^\infty (M,P\times Q).\tag{$\sharp $}\label{homeo}
\end{equation}

The homeomorphism (\ref{homeo}) however does not mean the singularity theoretic equivalence.
More specifically, isotopies of $F$ and $G$ do not always induce an isotopy of $(F,G)$, and an isotopy of $\varphi $ does not always induce isotopies of $\pi _P\circ \varphi $ and $\pi _Q\circ \varphi $.
Here, an isotopy of a map is a homotopy preserving the topological properties of the map.
The partition of a mapping space into isotopy classes is of general interest in singularity theory, but few things are known about the relation between the partitions of the both sides of the homeomorphism (\ref{homeo}).

We focus on the case where $M$ is closed and $3$-dimensional, and $P,Q$ are $1$-dimensional.
We do not assume the orientability of $M$.
Suppose $F:M\rightarrow P$ and $G:M\rightarrow Q$ are smooth functions such that $\varphi =(F,G)$ is stable.
A singular point of $\varphi $ is then either a fold point or a cusp point.
By Levine's \cite{levine1} theorem, we can eliminate the cusp points by a homotopy of $\varphi $.
It implies that we can eliminate the cusp points by homotopies of $F$ and $G$.
Note that we cannot reduce the number of cusp points by an isotopy of $\varphi $.
We propose the following question.

\begin{Ques}\label{ques}
Can we eliminate the cusp points of $\varphi $ by (quasi-)isotopies of $F$ and $G$?
\end{Ques}

Our strategy to attack Question~\ref{ques} is to deform $\varphi $ by a sequence of global isotopies and local homotopies which can be realized by (quasi-)isotopies of $F,G$.
Johnson \cite[Section~6]{johnson1} showed what kind of global isotopy of $\varphi $ can be realized by (quasi-)isotopies of $F,G$ (see Corollary~\ref{isotopy}).
In this paper, we prove the following.

\begin{Thm}\label{moves}
The local moves of the discriminant set of $\varphi $ as in Figure~\ref{fig_moves} can be realized by isotopies of $F$ and $G$.
\end{Thm}

\begin{figure}[ht]
\begin{picture}(215,185)(0,0)
\put(0,105){\includegraphics{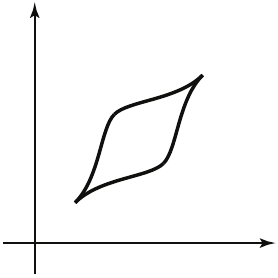}}
\put(135,105){\includegraphics{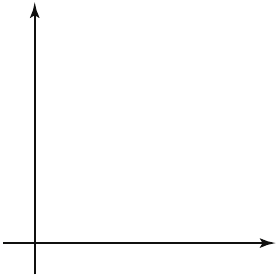}}
\put(0,5){\includegraphics{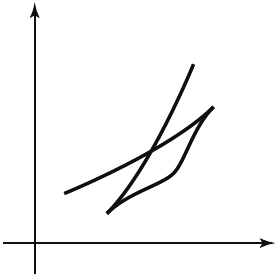}}
\put(135,5){\includegraphics{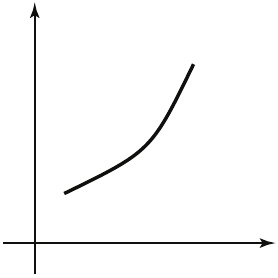}}
\put(95,40){{\LARGE $\longrightarrow $}}
\put(95,140){{\LARGE $\longrightarrow $}}
\put(0,172){$f$}
\put(135,172){$f$}
\put(0,72){$f$}
\put(135,72){$f$}
\put(70,105){$g$}
\put(205,105){$g$}
\put(70,5){$g$}
\put(205,5){$g$}
\end{picture}
\caption{Local moves which reduce the number of cusp points of $\varphi $.
Here, $(f,g)$ is the coordinate system given by the product structure of $P\times Q$.
We require that the local moves do not involve tangent lines of the discriminant set parallel to the axes.}
\label{fig_moves}
\end{figure}
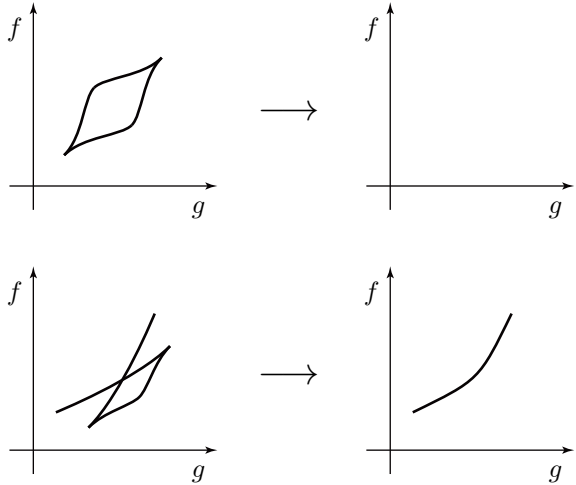

\noindent
Forgetting the axes, these moves are ones of generic local moves of singularities of $\varphi $ known as a ``lip" and a ``swallow-tail", respectively.
See \cite{mata-lorenzo} for the classification of generic local moves.
We expect that we can use our method (Proposition~\ref{cerf}) to work out other local moves, and we hope that we can use them to approach a global theory.

We would like to mention here the relation of this work to Heegaard theory of $3$-manifolds.
Rubinstein--Scharlemann \cite{rubinstein-scharlemann1} introduced {\it the graphic} for comparing two Heegaard splittings.
Kobayashi--Saeki \cite{kobayashi-saeki} interpreted the graphic as the discriminant set of the product map of two functions representing the splittings.
Johnson \cite{johnson1} gave an upper bound for the {\it Reidemeister--Singer distance} between two Heegaard splittings in terms of the graphic.
The author \cite{takao3} developed Johnson's idea to show that the Reidemeister--Singer distance is at most the sum of the genera of the splittings plus the number of cusp points of the product map.
If Question~\ref{ques} is answered positively, it ensures that the Reidemeister--Singer distance is at most the sum of the genera, which is the best possible bound by Hass--Thompson--Thurston's \cite{hass-thompson-thurston} example.

\subsection*{Acknowledgement}

The author would like to thank Osamu~Saeki, Takashi~Nishimura and Kentaro~Saji for valuable discussions and conversations.
He would also like to thank the referee for many helpful suggestions.

\section{Morse functions and stable maps}\label{Stable}

In this section, we briefly review standard definitions and facts on singularities of smooth maps.
We refer the reader to \cite{milnor} for basic notions in Morse theory, and to \cite{golubitsky-guillemin} for detailed description of stable maps.

A {\it Morse function} on a compact smooth manifold $M$ possibly with boundary is a smooth function $F$ from $M$ to either $\mathbb{R}$ or $S^1$ satisfying the following:
\begin{itemize}
\item All the critical points of $F$ are non-degenerate and belong to ${\rm int}M$.
\item The function $F$ is constant on each component of $\partial M$.
\end{itemize}
We regard $\mathbb{R}$ and $S^1$ as oriented.
This gives the distinction between locally minimal components and locally maximal components of $\partial M$ with respect to $F$.
A smooth homotopy consisting of Morse functions is called a {\it quasi-isotopy}.
Two Morse functions are said to be {\it quasi-isotopic} if there is a quasi-isotopy between them.

We consider generic homotopies of smooth functions from a compact connected surface $\Sigma $ to ${\mathbb R}$.
Let $\{ \alpha_t:\Sigma \rightarrow {\mathbb R}\} _{t\in [-1,1]}$ be a smooth homotopy.
A {\it birth} (resp. {\it death}) of $\{ \alpha_t\} _{t\in [-1,1]}$ is the pair $(o,\sigma )$ of $o\in (-1,1)$ and $\sigma \in {\rm int}\Sigma $ such that $\alpha_\tau (\xi ,\eta )=\tau \xi _\tau -\xi _\tau ^3+\eta _\tau ^2$ for a local coordinate system $(\xi ,\eta )$ at $\sigma $, a local coordinate $\tau $ at $o$ whose direction agrees (resp. disagrees) with that of $t$, and a local coordinate of the target ${\mathbb R}$.
Here $\{ (\xi ,\eta )\mapsto (\xi _\tau ,\eta _\tau )\} _\tau $ is a smooth family of coordinate transformations.
A {\it passing} of $\{ \alpha_t\} _{t\in [-1,1]}$ is the pair $(o,\{ \sigma ,\sigma '\} )$ of $o\in (-1,1)$ and $\{ \sigma ,\sigma '\} \subset {\rm int}\Sigma $ such that $\sigma ,\sigma '$ are non-degenerate critical points of $\alpha_o$ with the same value, and $\{ (t,v)\in {\mathbb R}^2\mid v\text{ is a critical value of }\alpha_t|_{U\cup U'}\} $ has a transverse crossing at $(o,\alpha_o(\sigma ))$ for small neighborhoods $U,U'$ of $\sigma ,\sigma '$, respectively.
The homotopy $\{ \alpha_t\} _{t\in [-1,1]}$ is said to be {\it generic} if it consists of Morse functions whose critical points have pairwise distinct values except that $\{ \alpha _t\} _{t\in [-1,1]}$ has either a single birth, a single death or a single passing at each $t$ in a finite subset of $(-1,1)$.
Note that a generic homotopy without births and deaths is a quasi-isotopy.

\begin{Thm}[Maksymenko \cite{maksymenko}]\label{maksymenko}
Two Morse functions from $\Sigma $ to ${\mathbb R}$ are quasi-isotopic if and only if they have the same number of critical points of each index, and the same sets of locally minimal and locally maximal components of $\partial \Sigma $.
\end{Thm}

An {\it isotopy} of a smooth map $\varphi :M\rightarrow N$ between general smooth manifolds $M,N$ is a homotopy $\{ \varphi _t:M\rightarrow N\} _{t\in [0,1]}$ which is decomposed as $\varphi _t=H^N_t\circ \varphi \circ H^M_t$.
Here $\{ H^M_t\} _{t\in [0,1]}$, $\{ H^N_t\} _{t\in [0,1]}$ are smooth ambient isotopies of $M,N$, respectively, such that $H^M_0=id_M$, $H^N_0=id_N$.
Two smooth maps are said to be {\it isotopic} if there is an isotopy between them.

A {\it stable map} from $M$ to $N$ is a smooth map $\varphi :M\rightarrow N$ such that there exists an open neighborhood $U$ of $\varphi $ in $C^\infty (M,N)$ such that every map in $U$ is isotopic to $\varphi $.
We remark that an equivalent definition of stable map is given by using ``{\it right-left equivalent}'' in place of ``isotopic''.
We note that, in the case where $M$ is closed and $N$ is either $\mathbb{R}$ or $S^1$, the smooth map $\varphi $ is stable if and only if $\varphi $ is a Morse function whose critical points have pairwise distinct values.

Consider the case where $M$ is a closed smooth $3$-manifold and $N$ is a smooth surface.
Recall that $p\in M$ is a {\it regular point} of $\varphi $ if the differential $(d\varphi )_p:T_pM\rightarrow T_{\varphi (p)}N$ is surjective, and otherwise a {\it singular point}.
The set $S_\varphi $ of singular points of $\varphi $ is called the {\it singular set} and its image $\varphi (S_\varphi )$ is called the {\it discriminant set} of $\varphi $.
At a regular point $p\in M\setminus S_\varphi $, the map $\varphi $ has the standard form $\varphi (u,x,y)=(u,x)$ for some coordinate neighborhoods of $p$ and $\varphi (p)$.
Standard forms are also known for generic types of singular points as follows.

A {\it fold point} is a singular point $p$ where $\varphi $ has the form $\varphi (u,x,y)=(u,x^2\pm y^2)$ for a coordinate neighborhood $U$ of $p=(0,0,0)$ and a coordinate neighborhood of $\varphi (p)=(0,0)$.
The Jacobian matrix of $\varphi (u,x,y)=(u,x^2\pm y^2)$ says that the singular set $S_\varphi \cap U$ is the arc $\{ (u,0,0)\} $.
It follows that each singular point on $\{ (u,0,0)\} $ is also a fold point by a translation of the local coordinates.
The arc $\{ (u,0,0)\} $ is embedded to the arc $\{ (u,0)\} \subset N$ by $\varphi $.

A {\it cusp point} is a singular point $p$ where $\varphi $ has the local form $\varphi (u,x,y)=(u,ux-x^3+y^2)$.
One can check that the singular set $S_\varphi \cap U$ is the arc $\{ (3x^2,x,0)\} $, and consists of fold points except for the cusp point $p=(0,0,0)$.
Note that the arc $\{ (3x^2,x,0)\} $ is a regular curve but its image $\{ (3x^2,2x^3)\} \subset N$ has a cusp at $\varphi (p)=(0,0)$.

Assume that the singular set $S_\varphi $ consists only of fold points and cusp points.
By the above local observations and the compactness of $M$, we can see the outline of $S_\varphi $.
It is a $1$-dimensional submanifold of $M$, namely a collection of smooth circles.
There are finitely many cusp points and the restriction $\varphi |_{S_\varphi }$ is an immersion except that each cusp point maps to a cusp.
The next characterization of stable maps follows from Mather's theorems \cite[Theorem~A, Proposition~1.8]{mather4} and \cite[Theorem~4.1]{mather5}.

\begin{Thm}[Mather]\label{mather}
A smooth map $\varphi $ from a closed smooth $3$-manifold to a smooth surface is stable if and only if:
\begin{itemize}
\item The singular set $S_\varphi $ consists only of fold points and cusp points.
\item The restriction $\varphi |_{S_\varphi }$ has no double points at cusps, and the immersion $\varphi |_{S_\varphi \setminus \text{(cusp points)}}$ has only normal crossings.
\end{itemize}
\end{Thm}

The {\it Stein factorization} $W_\varphi $ of a general smooth map $\varphi :M\rightarrow N$ is the quotient space $M/\sim $, where $p_1\sim p_2$ if $p_1,p_2$ belong to the same connected component of a level set of $\varphi $.
Let $q_\varphi $ denote the quotient map form $M$ to $W_\varphi $.
We can see that there is also a unique continuous map $\bar{\varphi }:W_\varphi \rightarrow N$ such that $\varphi =\bar{\varphi }\circ q_{\varphi }$.
The Stein factorization of a stable map from a closed smooth $3$-manifold to a smooth surface is, in fact, a $2$-dimensional cell complex.
See \cite{levine2} for example.

\section{Two functions and the product map}

In this section, we review a local theory of singularities of two functions and the product map.

We use the following notation.
Suppose $M$ is a closed smooth $3$-manifold and $P,Q$ are either ${\mathbb R}$ or $S^1$.
Let $F:M\rightarrow P$ and $G:M\rightarrow Q$ be smooth functions, and $\varphi $ denote the product map $(F,G)$.
While we do not assume $F,G$ to be Morse, we assume $\varphi $ to be stable in this section.

The singular set $S_\varphi $ includes the critical points of $F$ and $G$, which can be seen as follows.
For each point $p\in M$, there is a local coordinate system $(f,g)$ at $\varphi (p)$ given by the product structure of $P\times Q$.
The Jacobian matrix of $\varphi $ with respect to this coordinate system is composed of the gradients of $F$ and $G$.
If $p$ is a critical point of $F$ or $G$, the Jacobian matrix has rank at most one, namely $p$ is a singular point of $\varphi $.

We read information about $F,G$ from the discriminant set of $\varphi $.
Note that $\varphi |_{S_\varphi }$ is an immersion of circles with finitely many cusps.
We can define the {\it slope} of the discriminant set $\varphi (S_\varphi )$ at $\varphi (p)$ for each $p\in S_\varphi $ with respect to the coordinate system $(f,g)$.
In particular, a point on the discriminant set with slope zero (resp. infinity) is called a {\it horizontal} (resp. {\it vertical}) {\it point}.
We can also define the second derivative of $\varphi (S_\varphi )$ outside of vertical points and cusps, by regarding an arc of $\varphi (S_\varphi )$ as the graph of a function.
In particular, a point with second derivative zero is called an {\it inflection point}.
Since zero or non-zero of the second derivative is preserved by rotating the coordinate system, an inflection point can be defined also for vertical points.

\begin{Lem}\label{critical}
A point $p\in M$ is a critical point of $F$ (resp. $G$) if and only if $\varphi (p)$ is a vertical (resp. horizontal) point of the image of a small neighborhood of $p$ in $S_\varphi $.
\end{Lem}

\begin{Lem}\label{degeneracy}
A critical point $p$ of $F$ (resp. $G$) degenerates if and only if $p$ is a fold point of $\varphi $ and $\varphi (p)$ is a vertical (resp. horizontal) inflection point of the image of a small neighborhood of $p$ in $S_\varphi $.
\end{Lem}

The above lemmas were originally described by Johnson \cite[Lemmas 10 and 11]{johnson1}, and simple analytic proofs were given by the author \cite[Lemmas 11 and 12]{takao3}.
Both Johnson and the author considered only the case of $P=Q={\mathbb R}$, but the proofs are independent of whether $P,Q$ are ${\mathbb R}$ or $S^1$.

By Lemmas \ref{critical} and \ref{degeneracy}, the function $F$ (resp. $G$) is Morse if the discriminant set $\varphi (S_\varphi )$ does not have vertical (resp. horizontal) inflection points.
Note that the $f$-coordinate ($g$-coordinate) of each vertical (resp. horizontal) point of $\varphi (S_\varphi )$ corresponds to the critical value of $F$ (resp. $G$).
It follows that the Morse function $F$ (resp. $G$) is stable if $\varphi (S_\varphi )$ does not have vertical (resp. horizontal) double tangent lines.

\begin{Cor}\label{isotopy}
A deformation of $\varphi (S_\varphi )$ by an ambient isotopy of $P\times Q$ can be realized by isotopies of $F,G$ if it keeps $\varphi (S_\varphi )$ without horizontal or vertical, inflection points and double tangent lines.
\end{Cor}

\begin{proof}
Let $\{ H_t\} _{t\in [0,1]}$ be a smooth ambient isotopy of $P\times Q$ such that $H_0=id_{P\times Q}$.
By the definitions, $\{ H_t\circ \varphi \} _{t\in [0,1]}$ is an isotopy of $\varphi $ and consists of stable maps.
It induces homotopies $\{ F_t=\pi _P\circ H_t\circ \varphi \} _{t\in [0,1]}$ of $F$ and $\{ G_t=\pi _Q\circ H_t\circ \varphi \} _{t\in [0,1]}$ of $G$.
The deformed discriminant set $H_t(\varphi (S_\varphi ))$ is the discriminant set of $H_t\circ \varphi =(F_t,G_t)$ for each $t\in [0,1]$.
Since $H_t(\varphi (S_\varphi ))$ does not have horizontal or vertical, inflection points and double tangent lines, $F_t$ and $G_t$ are stable for each $t\in [0,1]$.
The homotopies $\{ F_t\} _{t\in [0,1]}$ and $\{ G_t\} _{t\in [0,1]}$ are therefore isotopies.
\end{proof}

\section{Restrictions to level surfaces}\label{Restrictions}

In this section, we consider the relation between a product map restricted to an appropriate domain and the family of the restrictions of one function to level surfaces of the other function.

We use the following notation again.
Suppose $M$ is a closed smooth $3$-manifold and $P,Q$ are either ${\mathbb R}$ or $S^1$.
Let $F:M\rightarrow P$ and $G:M\rightarrow Q$ be smooth functions, and $\varphi $ denote the product map $(F,G)$.
We do not assume $F,G$ to be Morse nor $\varphi $ to be stable at this stage.

We consider the restriction of $F$ to a level surface of $G$.
Suppose $p\in M$ is a regular point of $G$, and hence the level set $G^{-1}(G(p))$ is a regular surface near the point $p$.
The point $p$ is a critical point of $F|_{G^{-1}(G(p))}$ if and only if $p$ is a singular point of $\varphi $, which can be seen as follows.
The gradient of the restriction $F|_{G^{-1}(G(p))}$ is the projection of the gradient of $F$ to the orthogonal complement of the gradient of $G$.
It is zero if and only if the gradients of $F$ and $G$ are linearly dependent.
They are linearly dependent if and only if the differential $(d\varphi )_p$ is not surjective.

\begin{Lem}\label{fold}
A critical point $p$ of $F|_{G^{-1}(G(p))}$ is non-degenerate if and only if $p$ is a fold point of $\varphi $.
\end{Lem}

\begin{proof}
Since $p$ is a regular point of $G$, there is a local coordinate system $(\xi ,\eta ,\tau )$ of $M$ at $p=(0,0,0)$ such that $G(\xi ,\eta ,\tau )=\tau +G(p)$ and $(\xi ,\eta )$ is a local coordinate system of $G^{-1}(G(p))$ at $p$.
This gives us $\varphi (\xi ,\eta ,\tau )=(F(\xi ,\eta ,\tau ),\tau +G(p))$ and $F|_{G^{-1}(G(p))}(\xi ,\eta )=F(\xi ,\eta ,0)$.
By Morin's \cite[Lemme~1]{morin} characterization, the point $p$ is a fold point of $\varphi $ if and only if the critical point $p$ of $F(\xi ,\eta ,0)$ is non-degenerate.
\end{proof}

We consider the restrictions of $F$ to the level surfaces of $G$ in a domain $V\subset M$ which is defined as follows.
Note that there is a canonical covering ${\mathbb R}^2\rightarrow P\times Q$ by identifying $S^1$ with ${\mathbb R}/{\mathbb Z}$.
Let $\bar{R} \subset {\mathbb R}^2$ be the region $\{ (f,g)\in {\mathbb R}^2\mid h_-(g)\leq f\leq h_+(g),\ g_-\leq g\leq g_+\} $, where $g_-,g_+\in {\mathbb R}$ are constants such that $g_-<g_+$ and $h_-,h_+:{\mathbb R}\rightarrow {\mathbb R}$ are smooth functions such that $h_-(g)<h_+(g)$ for every $g\in [g_-,g_+]$.
We assume that $\bar{R} $ is embedded to $R\subset P\times Q$ by the covering map.
Let $V$ be a connected component of the preimage $\varphi ^{-1}(R)$.
From now on, we consider $\varphi ,F,G$ restricted to $V$, which allows us to assume that $P=Q={\mathbb R}$.
We assume that $G$ does not have critical points in $V$, and that the discriminant set of $\varphi $ does not intersect the two edges $\{ f=h_-(g),h_+(g),\ g_-\leq g\leq g_+\} $ of $R$.
Each level set $G^{-1}(g)\cap V$ is then a regular surface whose boundaries are regular level curves of $F|_{G^{-1}(g)}$.
The space $V$ is therefore a $\Sigma $-bundle over $[g_-,g_+]$, namely the direct product $\Sigma \times [g_-,g_+]$.
Here $\Sigma $ is a compact connected surface and each $\Sigma \times \{ g\} $ is the level surface $G^{-1}(g)\cap V$.

The restrictions of $F$ to the level surfaces of $G$ determine a homotopy $\{ \alpha _t:\Sigma \rightarrow {\mathbb R}\} _{t\in [g_-,g_+]}$.
That is to say, $\alpha _t(\sigma )=F(\sigma ,t)$ for each point $(\sigma ,t)$ in $\Sigma \times [g_-,g_+]=V$.
The range of each $\alpha _t$ is contained in $[h_-(t),h_+(t)]$, and each component of $\partial \Sigma $ is either at the minimal level $h_-(t)$ or at the maximal level $h_+(t)$.
In particular, $\{ \alpha _t\} _{t\in [g_-,g_+]}$ preserves the sets of locally minimal and locally maximal components of $\partial \Sigma $.
By Lemma~\ref{fold} and the definition of a quasi-isotopy, we have the following.

\begin{Cor}\label{fold_cor}
The homotopy $\{ \alpha _t\} _{t\in [g_-,g_+]}$ is a quasi-isotopy if and only if the singular set $S_\varphi \cap V$ consists only of fold points.
\end{Cor}

\noindent
The ``only if" direction of this corollary extends to the following.

\begin{Lem}\label{stable}
If the homotopy $\{ \alpha _t\} _{t\in [g_-,g_+]}$ is generic, the map $\varphi $ is stable in $V$.
\end{Lem}

\begin{proof}
At each birth or death, $\{ \alpha _t\} _{t\in [g_-,g_+]}$ has the local form $\alpha _\tau (\xi ,\eta )=\tau \xi _\tau -\xi _\tau ^3+\eta _\tau ^2$ with the notation of Section~\ref{Stable}.
Choosing a local coordinate system $(u,x,y)$ of $M$ as $u=\tau $, $x=\xi _\tau $, $y=\eta _\tau $, the map $\varphi $ has the local form $\varphi (u,x,y)=(u,ux-x^3+y^2)$, which is of a cusp point.
Taking this together with the ``only if" direction of Corollary~\ref{fold_cor}, the singular set $S_\varphi \cap V$ consists only of fold points and cusp points.
The conditions of a generic homotopy about the critical values imply the second condition in Theorem~\ref{mather}.
\end{proof}

The discriminant set $\varphi (S_\varphi \cap V)$ is the so-called {\it Cerf graphic} of $\{ \alpha _t\} _{t\in [g_-,g_+]}$.
That is to say, the intersection of $\varphi (S_\varphi \cap V)$ with each line $l_t=\{ (f,g)\in {\mathbb R}^2\mid g=t\} $ corresponds to the critical values of $\alpha _t$, and we can read from $\varphi (S_\varphi \cap V)$ how the critical values of $\alpha _t$ moves with $t$.

We can read more about the behavior of $\{ \alpha _t\} _{t\in [g_-,g_+]}$ from the Stein factorization $q_\varphi (V)$.
For a general homotopy $\{ \beta _t:\Sigma \rightarrow {\mathbb R}\} _{t\in [g_-,g_+]}$, we call the Stein factorization of the map $(\sigma ,t)\mapsto (\beta _t(\sigma ),t)$ from $\Sigma \times [g_-,g_+]$ to ${\mathbb R}^2$ the {\it Cerf complex} of $\{ \beta _t\} _{t\in [g_-,g_+]}$.
Note that for each $t\in [g_-,g_+]$, the intersection of the Cerf complex $q_\varphi (V)$ of $\{ \alpha _t\} _{t\in [g_-,g_+]}$ with the preimage $\bar{\varphi }^{-1}(l_t)$ is the Stein factorization $W_{\alpha _t}$ of $\alpha _t$, and that the composition $\pi _P\circ \bar{\varphi }|_{q_\varphi (V)\cap \bar{\varphi }^{-1}(l_t)}$ is the map $\bar{\alpha _t}:W_{\alpha _t}\rightarrow {\mathbb R}$.
Suppose $\varphi $ is stable in $V$ and $l_t$ is disjoint from cusps and crossing points of the discriminant set $\varphi (S_\varphi \cap V)$.
The function $\alpha _t$ is Morse by Lemma~\ref{fold}, and the critical points have pairwise distinct values.
The Stein factorization $W_{\alpha _t}=q_\varphi (V)\cap \bar{\varphi }^{-1}(l_t)$ is then a graph whose vertex has valence $1$, $2$ or $3$.
Here, points in $q_\varphi (S_\varphi \cap V)\cap \bar{\varphi }^{-1}(l_t)$ are considered as vertices of the graph $W_{\alpha _t}$.
We remark that $W_{\alpha _t}$ has no valence $2$ vertices if $M$ is orientable.
We can see that a vertex of valence $2$ or $3$ corresponds to an index $1$ critical point of $\alpha _t$.
Regarding $\bar{\alpha _t}$ as a height function, a locally minimal (resp. locally maximal) valence $1$ vertex corresponds to an index $0$ (resp. $2$) critical point of $\alpha _t$ except that those at the minimal level $h_-(t)$ (resp. the maximal level $h_+(t)$) correspond to minimal (resp. maximal) components of $\partial \Sigma $.

For example, consider the situation of the bottom left of Figure~\ref{fig_moves}.
We can choose a parallelogram $R$ as in the top of Figure~\ref{fig_swallow} after an appropriate isotopy of $\varphi (S_\varphi )$.
There exists a component $V$ of the preimage $\varphi ^{-1}(R)$ containing the two cusp points.
The left of the bottom four rows of Figure~\ref{fig_swallow} shows the possible structures of $q_\varphi (V)$, and the right shows the corresponding structures of $W_{\alpha _t}$ for $t=g_-,g^-,g^+,g_+$.
We remark that the structures as in the bottom row may not appear if $M$ is orientable.

\begin{figure}[ht]
\begin{picture}(270,470)(0,0)
\put(70,380){\includegraphics{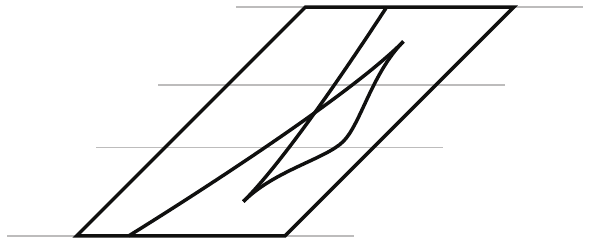}}
\put(243,444){$l_{g_+}$}
\put(221,422){$l_{g^+}$}
\put(203,403){$l_{g^-}$}
\put(179,378){$l_{g_-}$}
\put(0,285){\includegraphics{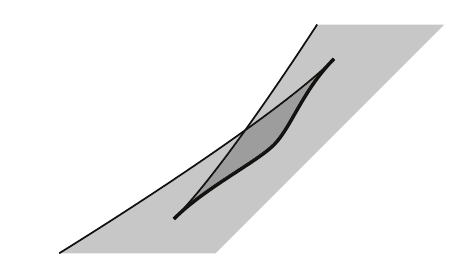}}
\put(0,190){\includegraphics{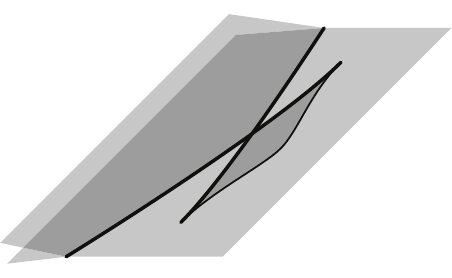}}
\put(0,95){\includegraphics{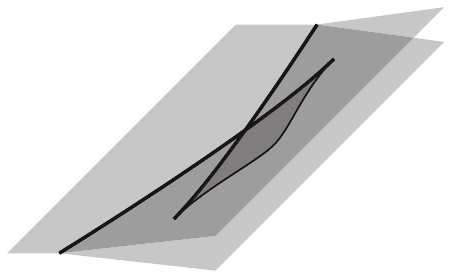}}
\put(0,0){\includegraphics{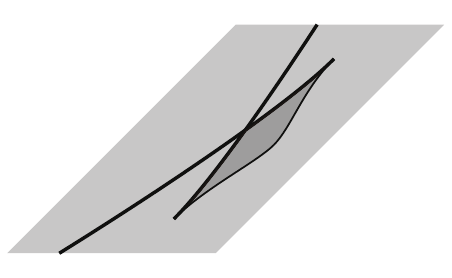}}
\put(140,285){\includegraphics{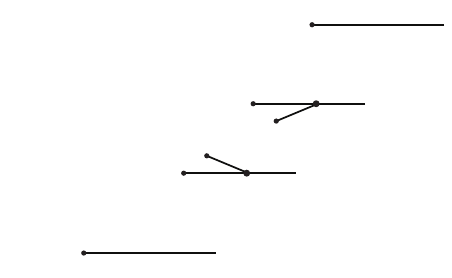}}
\put(140,190){\includegraphics{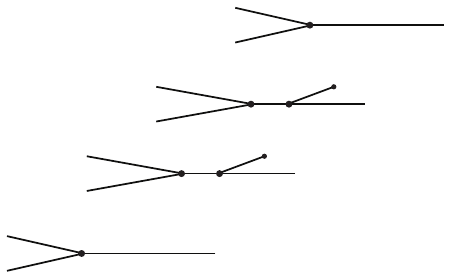}}
\put(140,95){\includegraphics{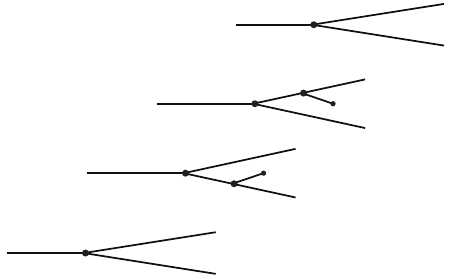}}
\put(140,0){\includegraphics{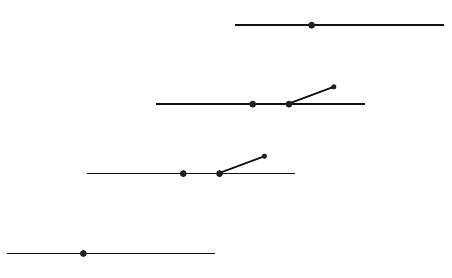}}
\end{picture}
\caption{A choice of $R$ and the four possible structures of $q_\varphi (V)$ and the corresponding structures of $W_{\alpha _t}$.}
\label{fig_swallow}
\end{figure}
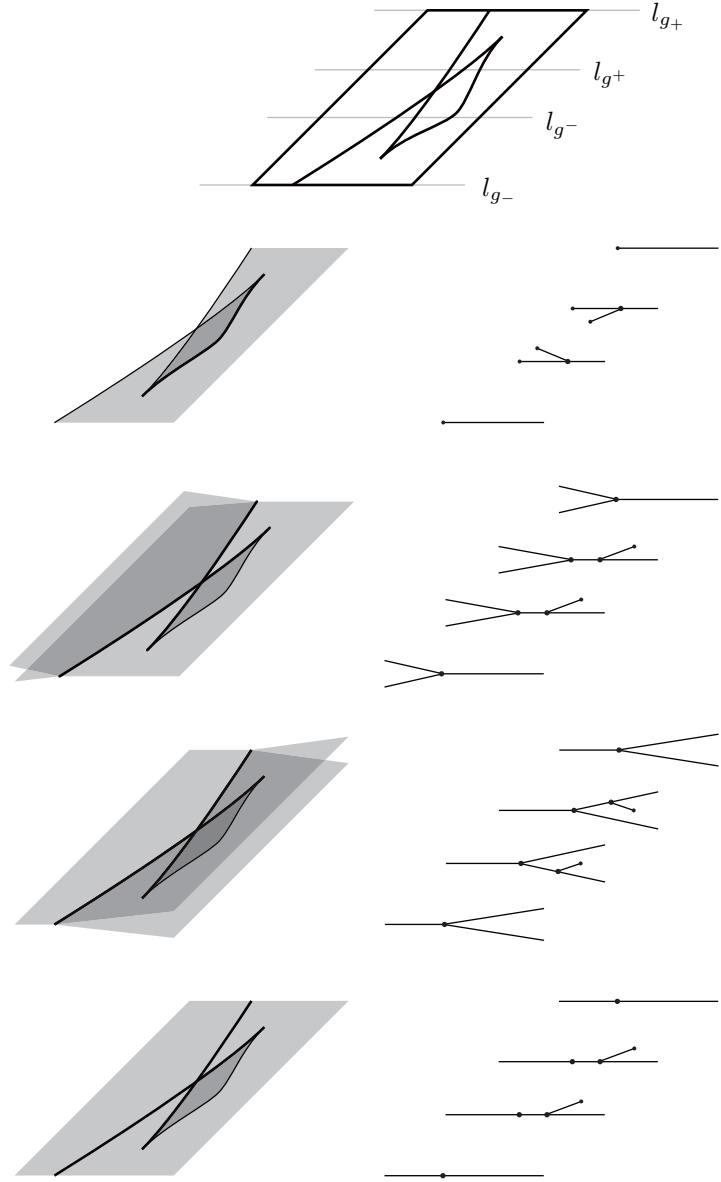

\section{Moves of product maps}

For the proof of Theorem~\ref{moves}, we make more general statements about moves of singularities of product maps.

We use the following notation.
Suppose $M$ is a closed smooth $3$-manifold and $P,Q$ are either ${\mathbb R}$ or $S^1$.
Let $F:M\rightarrow P$ and $G:M\rightarrow Q$ be smooth functions, and $\varphi $ denote the product map $(F,G)$.
Let $R\subset P\times Q$, $V\subset M$ and $\{ \alpha _t:\Sigma \rightarrow {\mathbb R}\} _{t\in [g_-,g_+]}$ be as described in Section~\ref{Restrictions}.
We consider $\varphi ,F,G$ restricted to $V$ and we can assume that $P=Q={\mathbb R}$.
We assume that $\varphi $ is stable in $V$, and the following:
\begin{enumerate}
\item The region $R$ is a parallelogram
\begin{equation*}
\left\{ (f,g)\in {\mathbb R}^2\mid f_-+a(g-g_-)\leq f\leq f_++a(g-g_-),\ g_-\leq g\leq g_+\right\} ,
\end{equation*}
where $f_-<f_+$, $g_-<g_+$, $a\in {\mathbb R}$ and $f_+<f_-+a(g_+-g_-)$.
\item The discriminant set $\varphi (S_\varphi \cap V)$ has neither horizontal points nor vertical points.
\item The Stein factorization $q_\varphi (V)$ has at least one edge which maps to one of the edges $\{ f=f_\pm +a(g-g_-),\ g_-\leq g\leq g_+\} $ of $R$.
\end{enumerate}

We can then regard a modification of the homotopy $\{ \alpha _t\} _{t\in [g_-,g_+]}$ as an isotopy of the function $F$ in the sense of the following proposition.

\begin{Prop}\label{cerf}
Let $\{ \beta _t:\Sigma \rightarrow {\mathbb R}\} _{t\in [g_-,g_+]}$ be a generic homotopy from $\beta _{g_-}=\alpha _{g_-}$ to $\beta _{g_+}=\alpha _{g_+}$.
Then there exists a smooth function $\tilde{F}$ isotopic to $F$ such that $\tilde{F}|_{M\setminus V}=F|_{M\setminus V}$, and $\tilde{\varphi }=(\tilde{F},G)$ is stable, and the Stein factorization $q_{\tilde{\varphi }}(V)$ is homeomorphic to the Cerf complex of $\{ \beta _t\} _{t\in [g_-,g_+]}$.
\end{Prop}

\begin{proof}
We can assume that $a=1$, $f_-=0$, $f_+=\frac{1}{3}$, $g_-=0$ and $g_+=1$ after an isotopy of $\varphi (S_\varphi )$ by the condition (1) and Corollary~\ref{isotopy}.
Let $\bar{\beta }_t(\sigma )=\beta _t(\sigma )-t$ for $t\in [0,1]$, $\sigma \in \Sigma $, and let $c=\left| {\rm inf}\left\{ \frac{\partial }{\partial t}\bar{\beta }_t(\sigma )\mid t\in [0,1], \ \sigma \in \Sigma \right\} \right| $.
We define a continuous homotopy $\{ \hat{\beta }_t\} _{t\in [0,1]}$ by
\begin{equation*}
\hat{\beta }_t(\sigma )=
\begin{cases}
\left( 1-\frac{6c+2}{2c+1}t\right) \bar{\beta }_0(\sigma )+\frac{4}{3}t & \left( t\in \left[ 0,\frac{1}{3}\right] \right) \\ \frac{1}{6c+3}\bar{\beta }_{3t-1}(\sigma )+t+\frac{1}{9} & \left( t\in \left[ \frac{1}{3},\frac{2}{3}\right] \right) \\ \left( \frac{6c+2}{2c+1}t-\frac{4c+1}{2c+1}\right) \bar{\beta }_1(\sigma )+\frac{2}{3}t+\frac{1}{3} & \left( t\in \left[ \frac{2}{3},1\right] \right) 
\end{cases}.
\end{equation*}

In the first interval $\left[ 0,\frac{1}{3}\right] $, the derivative $\frac{\partial }{\partial t}\hat{\beta }_t(\sigma )$ is positive since
\begin{equation*}
\frac{\partial }{\partial t}\hat{\beta }_t(\sigma )=-\frac{6c+2}{2c+1}\bar{\beta }_0(\sigma )+\frac{4}{3}=-\frac{6c+2}{2c+1}\alpha _0(\sigma )+\frac{4}{3}
\end{equation*}
and $0\leq \alpha _0(\sigma )\leq \frac{1}{3}$.
It is positive also in the last interval $\left[ \frac{2}{3},1\right] $ similarly.
In the middle interval, since
\begin{equation*}
\frac{\partial }{\partial t}\hat{\beta }_t(\sigma )=\frac{1}{6c+3}\frac{\partial }{\partial t}\bar{\beta }_{3t-1}(\sigma )+1
\end{equation*}
and
\begin{equation*}
-3c\leq \frac{\partial }{\partial t}\bar{\beta }_{3t-1}(\sigma )\leq 3c,
\end{equation*}
we have
\begin{equation*}
\frac{1}{2}<\frac{-3c}{6c+3}+1\leq \frac{\partial }{\partial t}\hat{\beta }_t(\sigma )\leq \frac{3c}{6c+3}+1<\frac{3}{2}.
\end{equation*}
The derivative $\frac{\partial }{\partial t}\hat{\beta }_t(\sigma )$ is thus positive for $t\in [0,1]$ except that the right and left derivatives may disagree at $t=\frac{1}{3}, \ \frac{2}{3}$.

Note that the ranges of $\hat{\beta }_0,\hat{\beta }_{\frac{1}{3}},\hat{\beta }_{\frac{2}{3}},\hat{\beta }_1$ are bounded as follows:
\begin{gather*}
\hat{\beta }_0(\sigma )=\bar{\beta }_0(\sigma )=\beta _0(\sigma )=\alpha _0(\sigma )\in \left[ 0,\frac{1}{3}\right] ,\\
\hat{\beta }_{\frac{1}{3}}(\sigma )=\frac{1}{6c+3}\bar{\beta }_0(\sigma )+\frac{4}{9}\in \left[ \frac{4}{9},\frac{1}{18c+9}+\frac{4}{9}\right] \subset \left[ \frac{4}{9},\frac{5}{9}\right] ,\\
\hat{\beta }_{\frac{2}{3}}(\sigma )=\frac{1}{6c+3}\bar{\beta }_1(\sigma )+\frac{7}{9}\in \left[ \frac{7}{9},\frac{1}{18c+9}+\frac{7}{9}\right] \subset \left[ \frac{7}{9},\frac{8}{9}\right] ,\\
\hat{\beta }_1(\sigma )=\bar{\beta }_1(\sigma )+1=\beta _1(\sigma )=\alpha _1(\sigma )\in \left[ 1,\frac{4}{3}\right] .
\end{gather*}
Since $\{ \hat{\beta }_t\} _{t\in [0,\frac{1}{3}]}$ connects $\hat{\beta }_0$ and $\hat{\beta }_{\frac{1}{3}}$ linearly, we can see that $\hat{\beta }_t(\sigma )\in \left[ t,\frac{1}{3}+t\right] $ for $t\in \left[ 0,\frac{1}{3}\right] $.
The same holds for $t\in \left[ \frac{2}{3},1\right] $.
We can see that the same holds also for $t\in \left[ \frac{1}{3},\frac{2}{3}\right] $ since $\hat{\beta }_{\frac{1}{3}}(\sigma )\in \left[ \frac{4}{9},\frac{5}{9}\right] $, $\hat{\beta }_{\frac{2}{3}}(\sigma )\in \left[ \frac{7}{9},\frac{8}{9}\right] $ and $\frac{1}{2}<\frac{\partial }{\partial t}\hat{\beta }_t(\sigma )<\frac{3}{2}$.
The range of $\hat{\beta }_t$ is thus contained in $\left[ t,\frac{1}{3}+t\right] $ for $t\in [0,1]$ as well as that of $\alpha _t$.

Note that the differential $(d\hat{\beta }_t)_\sigma $ is a scalar multiplication of $(d\beta _{t'})_\sigma $, where $t'=0$ if $t\in \left[ 0,\frac{1}{3}\right] $, $t'=3t-1$ if $t\in \left[ \frac{1}{3},\frac{2}{3}\right] $ or $t'=1$ if $t\in \left[ \frac{2}{3},1\right] $.
The homotopy $\{ \hat{\beta }_t\} _{t\in [0,1]}$ therefore keeps $\partial \Sigma $ without critical points as well as $\{ \beta _t\} _{t\in [0,1]}$.
In particular, $\{ \hat{\beta }_t\} _{t\in [0,1]}$ preserves the sets of locally minimal and locally maximal components of $\partial \Sigma $.
By deforming $\{ \hat{\beta }_t\} _{t\in [0,1]}$ in a collar neighborhood of $\partial \Sigma $, we can obtain a homotopy $\{ \tilde{\beta }_t\} _{t\in [0,1]}$ such that the locally minimal (resp. the locally maximal) components of $\partial \Sigma $ are at the level $t$ (resp. $\frac{1}{3}+t$) for each $t\in [0,1]$.

The homotopy $\{ \tilde{\beta }_t\} _{t\in [0,1]}$ then determines a continuous function $\tilde{F}:M\rightarrow P$.
That is to say, $\tilde{F}(\sigma ,t)=\tilde{\beta }_t(\sigma )$ for each point $(\sigma ,t)$ in $V=\Sigma \times [0,1]$ and $\tilde{F}|_{M\setminus V}=F|_{M\setminus V}$.
By arbitrarily small deformation in $V$, we can make $\{ \tilde{\beta }_t\} _{t\in [0,1]}$ generic and $\tilde{F}$ smooth keeping the differential $\frac{\partial }{\partial t}\tilde{F}(\sigma ,t)=\frac{\partial }{\partial t}\tilde{\beta }_t(\sigma )$ positive.
The map $\tilde{\varphi }=(\tilde{F},G)$ is then stable by Lemma~\ref{stable}.
The Stein factorization $q_{\tilde{\varphi }}(V)$ is homeomorphic to the Cerf complex of $\{ \beta _t\} _{t\in [g_-,g_+]}$ by the constructions of $\{ \bar{\beta }_t\} _{t\in [0,1]}$, $\{ \hat{\beta }_t\} _{t\in [0,1]}$ and $\{ \tilde{\beta }_t\} _{t\in [0,1]}$.

The condition (3) implies that $\Sigma $ has non-empty boundary, and hence $V=\Sigma \times [0,1]$ is a handlebody.
By the condition (2) and Lemma~\ref{critical}, the original function $F$ has no critical points in the handlebody $V$.
By $\frac{\partial }{\partial t}\tilde{F}(\sigma ,t)>0$, the new function $\tilde{F}$ also has no critical points in $V$.
The topologies of the level sets $F^{-1}(f)\cap V$ and $\tilde{F}^{-1}(f)\cap V$ change with $f$ according to singularities of $F|_{\partial V}$ and $\tilde{F}|_{\partial V}$, respectively.
Since $F|_{\partial V}$ and $\tilde{F}|_{\partial V}$ coincide, there is a homeomorphism of $V$ which takes each $F^{-1}(f)\cap V$ to $\tilde{F}^{-1}(f)\cap V$.
It is known that the canonical homomorphism from the mapping class group of a handlebody to the mapping class group of the boundary surface is injective.
It follows that the homeomorphism is isotopic to the identity, and so $\tilde{F}$ is isotopic to $F$.
\end{proof}

\begin{Cor}\label{eliminate}
Assume the following in addition to the above.
Then there exists a smooth function $\tilde{F}$ isotopic to $F$ such that $\tilde{F}|_{M\setminus V}=F|_{M\setminus V}$, and $\tilde{\varphi }=(\tilde{F},G)$ is a stable map without cusp points in $V$.
\begin{enumerate}
\setcounter{enumi}{3}
\item The discriminant set $\varphi (S_\varphi \cap V)$ has no cusps and no crossing points on the two edges $\{ f_-+a(g-g_-)\leq f\leq f_++a(g-g_-),\ g=g_-,g_+\} $ of $R$.
\item The intersections of the Stein factorization $q_\varphi (V)$ with the preimages $\bar{\varphi }^{-1}(l_{g_-})$, $\bar{\varphi }^{-1}(l_{g_+})$ have the same numbers of locally minimal valence $1$ vertices, valence $2$ or $3$ vertices and locally maximal valence $1$ vertices.
\end{enumerate}
\end{Cor}

\begin{proof}
By the condition (4) and Lemma~\ref{fold}, $\alpha _{g_-},\alpha _{g_+}$ are Morse functions whose critical points have pairwise distinct values.
By the condition (5) and the arguments in Section~\ref{Restrictions}, $\alpha _{g_-},\alpha _{g_+}$ have the same number of critical points of each index.
By Theorem~\ref{maksymenko}, there exists a quasi-isotopy $\{ \beta _t\} _{t\in [g_-,g_+]}$ from $\beta _{g_-}=\alpha _{g_-}$ to $\beta _{g_+}=\alpha _{g_+}$.
Corollary~\ref{eliminate} follows from the proposition and Corollary~\ref{fold_cor}.
\end{proof}

The local moves in Theorem~\ref{moves} are the simplest ones to which we can apply the above.
We can see that the choice of $R$ in Figure~\ref{fig_swallow} satisfies the conditions (1), (2) and (4).
We can also see that the structures of $q_\varphi (V)$ in Figure~\ref{fig_swallow} satisfies the conditions (3) and (5).
By Corollary~\ref{eliminate}, we can cancel the pair of cusp points, and the result is uniquely as in the bottom right of Figure~\ref{fig_moves}.
Similarly, we can obtain the local move of the top of Figure~\ref{fig_moves}.
This completes the proof of Theorem~\ref{moves}.

\end{document}